\documentclass[12pt]{elsarticle}
\usepackage[colorlinks=true, allcolors=blue]{hyperref} 
\usepackage{amsmath}
\usepackage{amsthm}
\usepackage{dsfont}
\usepackage{graphicx}

\newtheorem{theorem}{Theorem}[section]
\newtheorem{lemma}[theorem]{Lemma}

\newtheorem{remark}{Remark}

\usepackage{comment}

\usepackage{tikz}
\usepackage{tikz-cd}
\tikzcdset{scale cd/.style={every label/.append style={scale=#1},
cells={nodes={scale=#1}}}}

\usetikzlibrary{matrix}

\tikzset{ 
table/.style={
  matrix of nodes,
  row sep=-\pgflinewidth,
  column sep=-\pgflinewidth,
  nodes={rectangle,text width=2.4em,align=center},
  text depth=1.25ex,
  text height=2.5ex,
  nodes in empty cells
},
}

\usetikzlibrary{tikzmark}

\usepackage{tocloft}

\usepackage{amssymb}

\journal{Indagationes Mathematicae}

\begin{document}

\begin{frontmatter}



\title{On integral cohomology algebra of some oriented Grassmann manifolds}


\author[afil]{Milica Jovanovi\' c}
\ead{milica.jovanovic@matf.bg.ac.rs}

\affiliation[afil]{organization={University of Belgrade, Faculty of Mathematics},
            addressline={Studentski trg 16}, 
            city={Belgrade},
            postcode={11000}, 
            country={Serbia}}

\begin{abstract}
The integral cohomology algebra of $\widetilde G_{6,3}$ has been determined in the recent work of Kalafat and Yalçınkaya. We completely determine the integral cohomology algebra of $\widetilde G_{n,3}$ for $n=8$ and $n=10$. The main method used to describe these algebras is the Leray-Serre spectral sequence. We also illustrate this method by determining the integral cohomology algebra of $\widetilde G_{n,2}$ for $n$ odd.
\end{abstract}




\end{frontmatter}



\section{Introduction} 

The importance of Grassmann manifolds has been established for a quite some time. The oriented Grassmann manifold $\widetilde G_{n,k}$ is defined as a space of oriented $k$-dimensional subspaces of $\mathbb R^n$. As the name suggests, it is indeed a manifold of dimension $k(n-k)$. Recently, much work has been done in understanding the $\mathbb Z_2$ cohomology algebra of oriented Grassmann manifolds. For $k=2$ it was determined in \cite{korbas} and in \cite{basu} the authors obtain an almost complete description of this algebra for $k=3$ and $n$ close to a power of two. On the other hand, it seems that very little is known about the integral cohomology of $\widetilde G_{n,k}$, especially about the algebra structure. In a paper by Lai \cite{lai}, the cohomology algebra of $\widetilde G_{n,2}$ for $n$ even has been determined and this result was extended to all $n$ by Van\v zura in \cite{vanzura}. We give a different proof of Van\v zura's result for $n$ odd in Theorem \ref{G2n+1,2}. For $k\geqslant 3$ far less is known. In their recent work \cite{kalafat}, Kalafat and Yalçınkaya have managed to describe the integral cohomology algebra of $\widetilde G_{6,3}$ by using the special Lagrangian locus of $\widetilde G_{6,3}$. In this paper, using the Leray-Serre spectral sequence, we give the complete structure of the integral cohomology algebra of $\widetilde G_{n,3}$ for $n=8$ and $n=10$ in Theorems \ref{ringG83} and \ref{ringG103} respectively. We also give a few general facts about the integral cohomology algebra of $\widetilde G_{n,3}$. Since $\widetilde G_{n,k}\approx \widetilde G_{n,n-k}$, we will always assume $n\geqslant 2k$.

In Section \ref{Gn,2} we apply the Leray-Serre spectral sequence in order to obtain a description of $H^*(\widetilde G_{n,2};\mathbb Z)$, for $n$ odd. In Section \ref{W2,1} we determine the integral cohomology algebra of auxiliary space $W_{2,1}^n$ for even $n$. This space is defined in \cite{basu} as 
\[W_{2,1}^n=\left\{ (\widetilde P,v)\in \widetilde G_{n,2}\times S^{n-1} \mid \text{$\widetilde P$ and $v$ are orthogonal}\right\}.\]
We do this by applying the Leray-Serre spectral sequence on the sphere bundle $S^{n-3}\rightarrow W^n_{2,1}\xrightarrow{} \widetilde G_{n,2}$. In Section \ref{propertiesGn3} we give some general properties of the integral cohomology of $\widetilde G_{n,3}$. Finally, in Sections \ref{G83} and \ref{G103} we determine the algebras $H^*(\widetilde G_{8,3};\mathbb Z)$ and $H^*(\widetilde G_{10,3};\mathbb Z)$ with the application of Leray-Serre spectral sequence on fiber bundles $S^2\rightarrow W_{2,1}^n\rightarrow \widetilde G_{n,3}$ and $SO(3)\rightarrow V_{n,3}\rightarrow \widetilde G_{n,3}$ as our main tool ($V_{n,k}$ is the Stiefel manifold of orthonormal $k$-frames in $\mathbb R^n$).

\section{Cohomology algebra  $H^*( {\widetilde{G}}_{n,2};\mathbb Z)$}
\label{Gn,2}

The cohomology algebra $H^*(\widetilde G_{n,2};\mathbb Z)$ depends on the parity of $n$. If $n$ is even, then this algebra is determined in Theorem 2 in \cite{lai}. Specifically, the following result holds.

\begin{theorem}[Lai]\label{lai2n}
The integral cohomology groups of $\widetilde G_{2n,2}$ are isomorphic to those of $\mathbb CP^{n-1}\times S^{2n-2}$, namely, for $0\leqslant k\leqslant 4n-4$,
\[H^k(\widetilde G_{2n,2};\mathbb Z)\cong\begin{cases}
\mathbb Z,& \text{if $k$ is even and $k\ne 2n-2$},\\
\mathbb Z\oplus\mathbb Z,& \text{if }k=2n-2,\\
0,& \text{if }k\text{ is odd.}
\end{cases}\]
Let $\widetilde\gamma$ be the oriented tautological bundle over $\widetilde G_{2n,2}$ and $e(\widetilde\gamma)$ its Euler class. As an algebra, $H^*(\widetilde G_{2n,2};\mathbb Z)$ is generated by $\widetilde\Omega=e(\widetilde\gamma)\in H^2(\widetilde G_{2n,2};\mathbb Z)$ and $\kappa\in H^{2n-2}(\widetilde G_{2n,2};\mathbb Z)$ with the relation $\widetilde\Omega^{n}=2\kappa\widetilde\Omega$.
Moreover, if  $\widetilde \gamma^\perp$ is the orthogonal complement of $\widetilde \gamma$, $\Omega=e(\widetilde\gamma^\perp)\in H^{2n-2}(\widetilde G_{2n,2};\mathbb Z)$ its Euler class and $\mu\in H^{4n-4}(\widetilde G_{2n,2};\mathbb Z)$ the fundamental class, then we have the following relations
\[\begin{split}
&\Omega+\widetilde\Omega^{n-1}=2\kappa,~\kappa\widetilde\Omega^{n-1}=(-1)^n\mu,~\kappa\Omega=\mu,\\
&\kappa^2=\frac{1+(-1)^{n-1}}{2}\mu,~\widetilde\Omega^{2n-2}=2(-1)^n\mu,~\Omega^2=2\mu.
\end{split}\]

\end{theorem}

In the case of odd dimension, we have the following result. This was proved in \cite{vanzura}, but here we present a different proof, which relies on the Leray-Serre spectral sequence.

\begin{theorem} \label{G2n+1,2}
Let $n\geqslant 2$. The integral cohomology algebra of $\widetilde G_{2n+1,2}$ is
\[H^*(\widetilde G_{2n+1,2};\mathbb Z)=\frac{\mathbb Z[x_2,x_{2n}]}{\langle x_2^n-2x_{2n}, x_{2n}^2\rangle},\]
where $x_2\in H^2(\widetilde G_{2n+1,2};\mathbb Z)$ and  $x_{2n}\in H^{2n}(\widetilde G_{2n+1,2};\mathbb Z)$. If $\rho:\mathbb Z\rightarrow \mathbb Z_2$ is the modulo 2 map, then $\rho_*(x_2)=\widetilde w_2$ and $\rho_*(x_{2n})=a_{2n}$, where $\widetilde w_2\in H^2(\widetilde G_{2n,2};\mathbb Z_2)$ and $a_{2n}\in H^{2n}(\widetilde G_{2n,2};\mathbb Z_2)$ are generators from Theorem 2.1 in \cite{korbas}.
\end{theorem}
\begin{proof}
In order to determine the cohomology algebra $H^*(\widetilde G_{2n+1,2};\mathbb Z)$, we will use the Leray-Serre spectral sequence for the sphere bundle $S^1\rightarrow V_{2n+1,2}\xrightarrow{p} \widetilde G_{2n+1,2}$ where $p$ maps the ordered pair of orthogonal vectors to the oriented plane they span. The integral cohomology ring of Stiefel manifold $V_{2n+1,2}$ is well known (see \cite{mccleary} p. 153 or Theorem 2.3 in \cite{mimura}). We have
\[H^k(V_{2n+1,2};\mathbb Z)\cong 
\begin{cases}
\mathbb Z,& \text{for } k=0,4n-1\\
\mathbb Z_2,&\text{for } k=2n\\
0,& \text{otherwise}
\end{cases}\]
and the multiplication is trivial.

The only nontrivial differential in the spectral sequence is $d_2$. We will analyze the $E_2$ page from left and right and finish in the middle. Since $H^1(V_{2n+1,2};\mathbb Z)=0$, we have $E_2^{1,0}=0$ and also $E_2^{1,1}=0$. In similar fashion, we get that $E_2^{3,0}=E_2^{3,1}=0$ as $H^3(V_{2n+1,2};\mathbb Z)=0$. We can continue with same conclusions up to dimension $2n-1$, i.e., $E_2^{k,0}=E_2^{k,1}=0$ for $k$ odd between $1$ and $2n-1$. On the other hand, since $E_2^{0,1}\cong\mathbb Z$ and $H^k(V_{2n+1,2};\mathbb Z)=0$ for all even $k$ between $2$ and $2n-2$, we get that each $d_2:E_2^{k-2,1}\rightarrow E_2^{k,0}$ must be an isomorphism for $2\leqslant k\leqslant 2n-2$, i.e., $E_2^{k,0}\cong E_2^{k,1}\cong\mathbb Z$ for $k$ even between $0$ and $2n-2$.

Now starting from the right-hand side we first have $E_2^{4n-2,0} \cong \mathbb Z$ since $H^{4n-2}(\widetilde G_{2n+1,2};\mathbb Z)\cong \mathbb Z$ ($\widetilde G_{2n+1,2}$ is an orientable closed connected manifold of dimension $4n-2$) and then we have a series of isomorphisms
 \[E_2^{2n,0}\cong E_2^{2n,1}\xrightarrow{d_2} E_2^{2n+2,0} \cong E_2^{2n+2,1} \xrightarrow{d_2} \cdots \xrightarrow{d_2} E_2^{4n-2,0} \cong  \mathbb Z,\]
 so we get $E_2^{k,0}=E_2^{k,1}=\mathbb Z$ for $k$ even between $2n$ and $4n-2$. Similarly, $E_2^{4n-3,1} = 0$ since $H^{4n-2}(V_{2n+1,2};\mathbb Z) = 0$ and we get $E_2^{k,0}=E_2^{k,1}=0$ for $k$ odd between $2n+1$ and $4n-3$. The only nontrivial differential which is not an isomorphism is $d_2:E_2^{2n-2,1}\rightarrow E_2^{2n,0}$ and since $H^{2n}(V_{2n+1,2};\mathbb Z)=\mathbb Z_2$, this differential must be multiplication by $2$. The $E_2$ page of the spectral sequence is presented in Table \ref{tab:SVG2n+1}.

\begin{table}[h]
\begin{center}
\scalebox{0.96}{
\begingroup
\renewcommand{\arraystretch}{1.5} %
\begin{tabular}{c|c|c|c|c|c|c|c|c|c|c|c}
1 & $\mathbb Z$\tikzmark{a} & 0 & $\mathbb Z$ & $\cdots$ & $\mathbb Z$\tikzmark{c} & 0 & $\mathbb Z$ \tikzmark{e} & 0 & $\mathbb Z$ & $\cdots$ & $\mathbb Z$ \\\cline{1-12}
0 & $\mathbb Z$ & 0 & \tikzmark{b} $\mathbb Z$ & $\cdots$ & $\mathbb Z$ & 0 & \tikzmark{d}$\mathbb Z$ & 0 & \tikzmark{f}$\mathbb Z$ & $\cdots$ & $\mathbb Z$ \\
\cline{1-12}
 & 0 & 1 & 2 & $\cdots$ & $2n-2$ & $2n-1$ & $2n$ & $2n+1$ & $2n+2$ & $\cdots$ & $4n-2$ \\
  \end{tabular}

  \begin{tikzpicture}[overlay, remember picture, shorten >=.5pt, shorten <=.5pt, transform canvas={yshift=.25\baselineskip}]
    \draw [>->>] ([yshift=.75pt]{pic cs:a}) -- ({pic cs:b});
    \draw [->] ([yshift=.75pt]{pic cs:c}) -- ({pic cs:d}) node  [pos=0.35,above] {$~_{\cdot 2}$};
    \draw [>->>] ([yshift=.75pt]{pic cs:e}) -- ({pic cs:f});
  \end{tikzpicture}
\endgroup
}
\end{center}
\caption{\label{tab:SVG2n+1}The $E_2$ page for the sphere bundle $S^1\rightarrow V_{2n+1,2}\rightarrow \widetilde G_{2n+1,2}$}
\end{table}

So far we have determined cohomology groups of $\widetilde G_{2n+1,2}$ to be $\mathbb Z$ in even and $0$ in odd dimensions. Differential $d_2:E_2^{k,1}\rightarrow E_2^{k+2,0}$ is an isomorphism for $k=0$. If $s\in H^1(S^1;\mathbb Z)$ is the generator, let $x_2\in H^2(\widetilde G_{2n+1,2};\mathbb Z)$ be the generator such that $d_2(1\otimes s)=x_2\otimes 1$.
 From this relation we can determine all the other differentials $d_2$. If $k$ is odd, $d_2:E_2^{k,1}\rightarrow E_2^{k+2,0}$ is trivial and if $k$ is even and $x_k\in H^k(\widetilde G_{2n+1,2};\mathbb Z)$ then
\begin{center}
\scalebox{0.87}
{$d_2(x_k\otimes s)=d_2(x_k\otimes 1)(1\otimes s)+(-1)^k(x_k\otimes 1)d_2(1\otimes s)=(x_k\otimes 1)(x_2\otimes 1)=x_kx_2\otimes 1.$}
\end{center}

Therefore our algebra has two generators, $x_2$ and $x_{2n}\in H^{2n}(\widetilde G_{2n+1,2};\mathbb Z)$ and two relations $x_2^n=2x_{2n}$ and $x_{2n}^2=0$.

What remains is to prove that modulo two reduction of generators $x_2$ and $x_{2n}$ gives us generators of $H^*(\widetilde G_{2n+1,2};\mathbb Z_2)$. From Theorem 2.1 in \cite{korbas} we have
\[H^*(\widetilde G_{2n+1,2};\mathbb Z_2)=\mathbb Z_2[\widetilde w_2]/\langle\widetilde w_2^{n}\rangle\otimes \Lambda_{\mathbb Z_2}(a_{2n}).\]
The universal coefficient theorem gives us the following commutative diagram.
\[
\begin{tikzcd}[column sep = tiny, font = \small]
0
\arrow{r}
&
\mathrm{Ext}(H_{k-1}(\widetilde G_{2n+1,2}),\mathbb Z)
\arrow{d}{\rho_*}
\arrow{r}
&
H^k(\widetilde G_{2n+1,2};\mathbb Z)
\arrow{d}{\rho_*}
\arrow{r}
&
\mathrm{Hom}(H_k(\widetilde G_{2n+1,2}),\mathbb Z)
\arrow{d}{\rho_*}
\arrow{r}
&
0
\\
0
\arrow{r}
&
\mathrm{Ext}(H_{k-1}(\widetilde G_{2n+1,2}),\mathbb Z_2)
\arrow{r}
&
H^k(\widetilde G_{2n+1,2};\mathbb Z_2)
\arrow{r}
&
\mathrm{Hom}(H_k(\widetilde G_{2n+1,2}),\mathbb Z_2)
\arrow{r}
&
0
\end{tikzcd}
\]

For $k=2$, $H_{k-1}(\widetilde G_{2n+1,2})=0$, so previous diagram becomes

\[
\begin{tikzcd}[column sep = small]
0
\arrow{r}
&
H^2(\widetilde G_{2n+1,2};\mathbb Z)
\arrow{d}{\rho_*}
\arrow[two heads, tail]{r}
&
\mathrm{Hom}(H_2(\widetilde G_{2n+1,2}),\mathbb Z)
\arrow{d}{\rho_*}
\arrow{r}
&
0
\\
0
\arrow{r}
&
H^2(\widetilde G_{2n+1,2};\mathbb Z_2)
\arrow[two heads, tail]{r}
&
\mathrm{Hom}(H_2(\widetilde G_{2n+1,2}),\mathbb Z_2)
\arrow{r}
&
0
\end{tikzcd}
\]

Poincar\' e duality gives us that $H_2(\widetilde G_{2n+1,2})\cong\mathbb Z$, so the right homomorphism $\rho_*$ is an epimorphism and since horizontal homomorphisms are isomorphisms, $\rho_*:H^2(\widetilde G_{2n+1,2};\mathbb Z)\rightarrow H^2(\widetilde G_{2n+1,2};\mathbb Z_2)$ must also be an epimorphism, so $\rho_*(x_2)=\widetilde w_2$.
\\
\indent The situation is same for the case $k=2n$. We have $H_{2n-1}(\widetilde G_{2n+1,2})=0$, so we have smaller diagram as the previous one from which we conclude that $\rho_*:H^{2n}(\widetilde G_{2n+1,2};\mathbb Z)\rightarrow H^{2n}(\widetilde G_{2n+1,2};\mathbb Z_2)$ is an epimorphism as well, so $\rho_*(x_{2n})=a_{2n}$.
\end{proof}

\begin{remark} In the previous proof we have seen that differential $d_2$ in Table \ref{tab:SVG2n+1} is multiplication with $x_2$, so this must be the Euler class of the fiber bundle $S^1 \rightarrow V_{2n+1,2} \rightarrow \widetilde G_{2n+1,2}$. But this bundle is easily identified with the sphere bundle associated to the tautological bundle $\widetilde\gamma$, so $x_2=\widetilde \Omega=e(\widetilde\gamma)$, as in Theorem \ref{lai2n}. On the other hand, $\Omega=e(\widetilde\gamma^\perp)=0$, since $\Omega\in H^{2n-1}(\widetilde G_{2n+1,2},\mathbb Z)=0$.
\end{remark}

\section{Cohomology algebra  $H^*(W^{2n}_{2,1};\mathbb Z)$}
\label{W2,1}

The space $W^n_{2,1}$ is defined in \cite{basu} as
\[W_{2,1}^n=\left\{ (\widetilde P,v)\in \widetilde G_{n,2}\times S^{n-1} \mid \text{$\widetilde P$ and $v$ are orthogonal}\right\}.\]

For $n$ even, we will determine the integral cohomology algebra of this space using the Leray-Serre spectral sequence for the sphere bundle $S^{n-3}\rightarrow W^n_{2,1}\xrightarrow{p_1} \widetilde G_{n,2}$ where $p_1$ is the projection onto the first coordinate. Note that this is actually the sphere bundle associated to the orthogonal complement $\widetilde \gamma^\perp$  of the oriented tautological bundle over $\widetilde G_{n,2}$.

\begin{theorem}\label{W2n2,1}
Let $n\geqslant 3$. The integral cohomology algebra of $W^{2n}_{2,1}$ is
\[H^*(W^{2n}_{2,1};\mathbb Z)=\frac{\mathbb Z[\bar x_2,\bar x_{2n-2},\bar x_{2n-1}]}{\langle \bar x_2^{n-1}-2\bar x_{2n-2},\bar x_{2n-2}^2, \bar x_{2n-1}^2\rangle},\]
where $\bar x_2\in H^2(W^{2n}_{2,1};\mathbb Z)$, $\bar x_{2n-2}\in H^{2n-2}(W^{2n}_{2,1};\mathbb Z)$ and  $\bar x_{2n-1}\in H^{2n-1}(W^{2n}_{2,1};\mathbb Z)$.
\end{theorem}
\begin{proof} Theorem \ref{lai2n} gives us the $E_2$ page of the Leray-Serre spectral sequence for the sphere bundle $S^{2n-3}\rightarrow W_{2,1}^{2n}\rightarrow \widetilde G_{2n,2}$. In each cell of Table \ref{tab:SWG2n} are generators of this $E_2$ page, where $\kappa$ and $\widetilde \Omega$ are those generators from Theorem \ref{lai2n}. For instance, $\widetilde \Omega^2$ in cell $(4,0)$ means $E_2^{4,0}\cong\mathbb Z\langle\widetilde \Omega^2\rangle$ and $\widetilde \Omega^{n-1},~\kappa$ in cell $(2n-2,0)$ means $E_2^{2n-2,0}\cong\mathbb Z\langle\widetilde \Omega^{n-1}\rangle\oplus\mathbb Z\langle\kappa\rangle$. 

\begin{center}
\begin{table}[h]
\scalebox{0.8}{
\renewcommand{\arraystretch}{1.5} %
\begin{tabular}{c|c|c|c|c|c|c|c|c|c|c|c|c|c|c|c}
$2n-3$ & $1$ &  & $\widetilde\Omega$ & & $\widetilde\Omega^2$ & $\cdots$ & $\widetilde\Omega^{n-2}$ & & $\widetilde\Omega^{n-1},~\kappa$ & & $\kappa\widetilde\Omega$ & $\cdots$ & $\kappa\widetilde\Omega^{n-2}$ & & $\kappa \widetilde\Omega^{n-1}$ \\\cline{1-16}
$\vdots$ &&&&&&&&&&&&&&&\\\cline{1-16}
$0$ & $1$ &  & $\widetilde\Omega$ & & $\widetilde\Omega^2$ & $\cdots$ & $\widetilde\Omega^{n-2}$ & & $\widetilde\Omega^{n-1},~\kappa$ & & $\kappa\widetilde\Omega$ & $\cdots$ & $\kappa\widetilde\Omega^{n-2}$ & & $\kappa \widetilde\Omega^{n-1}$ \\\cline{1-16}
 & 0 &  & 2 &  & 4 & $\cdots$ & $2n-4$ & & $2n-2$ & & $2n$ &  $\cdots$ & $4n-6$ & & $4n-4$ \\
  \end{tabular}
}
\caption{\label{tab:SWG2n}The $E_2$ page for the sphere bundle $S^{2n-3}\rightarrow W_{2,1}^{2n}\rightarrow\widetilde G_{2n,2}$}
\end{table}
\end{center}

The only nontrivial differential is $d_{2n-2}$ and we know that it is multiplication by $\Omega$, since $\Omega$ is the Euler class of this sphere bundle. By multiplication we mean the following. Generators of $E_{2n-2}^{k,2n-3}=E_2^{k,2n-3}$ are  tensor products of elements in $H^k(\widetilde G_{2n,2};\mathbb Z)$ and a generator $s\in H^{2n-3}(S^{2n-3};\mathbb Z)$. For instance, by $\widetilde \Omega\in E_2^{2,2n-3}$ we mean $\widetilde \Omega \otimes s$. Similarly, $\widetilde \Omega\in E_2^{2,0}$ is actually $\widetilde \Omega\otimes 1$. Claim that $d_{2n-2}$ is multiplication by $\Omega$ only means that for $x\in H^k(\widetilde G_{2n,2};\mathbb Z)$ we have $d_{2n-2}(x\otimes s)=x\Omega\otimes 1$. Since $d_{2n-2}$ is the only nontrivial differential, we have $E_2=E_3=\cdots=E_{2n-2}$ and $E_{2n-1}=E_{2n}=\cdots=E_\infty$. Since $\widetilde \Omega \Omega=0$, we get $d_{2n-2}:E_{2n-2}^{k,2n-3}\rightarrow E_{2n-2}^{k+2n-2,0}$ is zero except at most for $k=0$ and $k=2n-2$. We know from Theorem \ref{lai2n} that $\Omega=2\kappa-\widetilde\Omega^{n-1}$, so $\mathrm{im}\{d_{2n-2}:E_{2n-2}^{0,2n-3}\rightarrow E_{2n-2}^{2n-2,0}\}=\mathbb Z\langle 2\kappa-\widetilde\Omega^{n-1}  \rangle$ and so we get
\[E_\infty^{2n-2,0}=E_{2n-2}^{2n-2,0}\big/\mathrm{im }~d_{2n-2}=\frac{\mathbb Z\langle\kappa\rangle\oplus\mathbb Z\langle\widetilde\Omega^{n-1}\rangle}{\mathbb Z\langle 2\kappa-\widetilde\Omega^{n-1} \rangle}\cong\mathbb Z\langle\kappa\rangle\]
and $E_\infty^{0,2n-3}=0$. Also from the same Theorem \ref{lai2n} we get that $d_{2n-2}:E_{2n-2}^{2n-2,2n-3}\rightarrow E_{2n-2}^{4n-4,0}$ is given on generators by
\[d_{2n-2}\widetilde \Omega^{n-1}=0,~d_{2n-2}\kappa=(-1)^n\kappa\widetilde \Omega^{n-1},\]
from which we get $E_\infty^{2n-2,2n-3}\cong \mathbb Z\langle\widetilde\Omega^{n-1}\rangle$ and $E_\infty^{4n-4,0}=0$.

Apart from these four, every other cell in $E_\infty$ page is the same as in $E_2$ page.

Now we can determine the cohomology algebra $H^*(W_{2,1}^{2n};\mathbb Z)$. We have an algebra isomorphism
\[H^*(W_{2,1}^{2n};\mathbb Z)\cong \mathrm{total}(E_\infty^{*,*}),\]
where $\mathrm{total}(E_\infty^{*,*})$ is the total complex (for definition see \cite{mccleary} p. 24).
 From $E_\infty$ we can conclude that $H^*(W_{2,1}^{2n};\mathbb Z)$ will have three generators $\bar x_2,~\bar x_{2n-2}$ and $\bar x_{2n-1}$ corresponding to $\widetilde \Omega\otimes 1$, $\kappa\otimes 1$ and $\widetilde \Omega \otimes s$ respectively and from relations in $H^*(\widetilde G_{2n,2};\mathbb Z)$ we get relations in $H^*(W_{2,1}^{2n};\mathbb Z)$. For instance, generator of $H^{4n-3}(W_{2,1}^{2n};\mathbb Z)$ is in correspondence with $\kappa\widetilde\Omega\otimes s=(\kappa\otimes 1)(\widetilde\Omega\otimes s)$ so it must be $\bar x_{2n-1}\bar x_{2n-2}$ and so on. After simple calculations we get the specified algebra structure.
\end{proof}

\section{Some general properties of $H^*(\widetilde G_{n,3};\mathbb Z)$}
\label{propertiesGn3}

\begin{lemma} \label{H01234} Let $n\geqslant 7$. For the first four and the last four integral cohomology groups of $\widetilde G_{n,3}$ the following relations hold:
\[\begin{matrix}
\begin{split}
&H^0(\widetilde G_{n,3};\mathbb Z)\cong\mathbb Z, \\
&H^1(\widetilde G_{n,3};\mathbb Z)=0, \\
&H^2(\widetilde G_{n,3};\mathbb Z)=0, \\
&H^3(\widetilde G_{n,3};\mathbb Z)\cong\mathbb Z_2,
\end{split}
&
\begin{split}
& H^{3(n-3)-3}(\widetilde G_{n,3};\mathbb Z)=0,\\
& H^{3(n-3)-2}(\widetilde G_{n,3};\mathbb Z)\cong\mathbb Z_2,\\
& H^{3(n-3)-1}(\widetilde G_{n,3};\mathbb Z)=0,\\
& H^{3(n-3)}(\widetilde G_{n,3};\mathbb Z)\cong\mathbb Z.
\end{split}
\end{matrix}\]
Also, the group $H^4(\widetilde G_{n,3};\mathbb Z)$ is free.
\end{lemma}
\begin{proof} Consider the homotopy exact sequence for the fiber bundle $SO(3)\rightarrow V_{n,3}\xrightarrow{p}\widetilde G_{n,3}$ where $p$ is given by the oriented span of the three vectors.
\[\begin{split}
\cdots&\rightarrow \pi_3(V_{n,3})\rightarrow \pi_3(\widetilde G_{n,3})\rightarrow \pi_2(SO(3))\rightarrow \pi_2(V_{n,3})\rightarrow \pi_2(\widetilde G_{n,3})\\
&\rightarrow\pi_1(SO(3))\rightarrow \pi_1(V_{n,3})\rightarrow \pi_1(\widetilde G_{n,3})\rightarrow 0
\end{split}\]

The Stiefel manifold $V_{n,3}$ is $(n-4)$-connected. Since, $n\geqslant 7$ the first three homotopy groups of $V_{n,3}$ are zero. Also, $\pi_1(SO(3))\cong\mathbb Z_2$ and $\pi_2(SO(3))=0$, so our sequence becomes
\[\cdots\rightarrow 0\rightarrow \pi_3(\widetilde G_{n,3})\rightarrow0\rightarrow 0\rightarrow \pi_2(\widetilde G_{n,3})\rightarrow \mathbb Z_2\rightarrow0\rightarrow \pi_1(\widetilde G_{n,3})\rightarrow 0.\]
It is now obvious that
\[\pi_1(\widetilde G_{n,3})=0,~~\pi_2(\widetilde G_{n,3})=\mathbb Z_2,~~\pi_3(\widetilde G_{n,3})=0.\]

Since $\widetilde G_{n,3}$ is $1$-connected, we have that Hurewicz homomorphism $h:\pi_k(\widetilde G_{n,3})\rightarrow H_k(\widetilde G_{n,3})$ is an isomorphism for $k\in\{1,2\}$ and an epimorphism for $k=3$ (see Theorem 10.25 in \cite{switzer}). This gives us that the first, the second and the third homology groups of $\widetilde G_{n,3}$ are $0$, $\mathbb Z_2$ and $0$ respectively.

The universal coefficient theorem claims that
\[H^k(\widetilde G_{n,3};\mathbb Z)\cong \mathrm{Ext}(H_{k-1}(\widetilde G_{n,3}),\mathbb Z)\oplus\mathrm{Hom}(H_k(\widetilde G_{n,3}),\mathbb Z).\]
From this relation we get the first four cohomology groups and also the fact that $H^4(\widetilde G_{n,3};\mathbb Z)$ is free. On the other hand, the Poincar\' e duality
\[H^{3(n-3)-k}(\widetilde G_{n,3};\mathbb Z)\cong H_k(\widetilde G_{n,3})\]
gives us the last four groups.
\end{proof}

\begin{remark}
The case $n=6$ is  covered in \cite{kalafat}, so the previous Lemma holds for all $n\geqslant 6$.
\end{remark}

Let us label the free and the torsion part of $H_k(\widetilde G_{n,3})$ as $F_k$ and $T_k$, i.e., $H_k(\widetilde G_{n,3})\cong F_k\oplus T_k$. The universal coefficient theorem gives us
\[H^k(\widetilde G_{n,3};\mathbb Z)\cong F_k\oplus T_{k-1}.\]

Since $\widetilde G_{n,3}\approx SO(n)/SO(3)\times SO(n-3)$, we can find the Poincar\' e polynomials for $\widetilde G_{n,3}$ in \cite{greub} (p. 494-496), so all groups $F_k$ are known.

Combining this with the Poincar\' e duality, which gives us
\[T_k\cong T_{3(n-3)-k-1},\]
the problem of determining integral cohomology groups of $\widetilde G_{n,3}$ becomes determining torsion parts $T_k$ for $4\leqslant k \leqslant \left[\frac{3(n-3)}{2}\right]$.

\section{Cohomology algebra $H^*(\widetilde G_{8,3};\mathbb Z)$}
\label{G83}

\begin{theorem}\label{ringG83}
The integral cohomology algebra of $\widetilde G_{8,3}$ is
\[H^*(\widetilde G_{8,3};\mathbb Z)\cong\frac{\mathbb Z[y_3,x_4,x_7]}{\langle 2y_3,y_3x_4,y_3^3,x_4^3,x_7^2\rangle},\]
where $y_3\in H^3(\widetilde G_{8,3};\mathbb Z)$, $x_4\in H^4(\widetilde G_{8,3};\mathbb Z)$ and  $x_7\in H^7(\widetilde G_{8,3};\mathbb Z)$.
\end{theorem}

\begin{proof} From \cite{greub} (p. 496), the Poincar\' e polynomial for $\widetilde G_{8,3}$ is
\[p(\widetilde G_{8,3})=1+x^4+x^7+x^8+x^{11}+x^{15},\]
so we know the free part of the cohomology groups of $\widetilde G_{8,3}$. When we apply results from Section \ref{propertiesGn3}, we see that integral cohomology groups of $\widetilde G_{8,3}$ are:
\begin{table}[h]
\scalebox{0.77}{
\renewcommand{\arraystretch}{1.5} %
\begin{tabular}{c|c|c|c|c|c|c|c|c|c|c|c|c|c|c|c|c}
$k$&0&1&2&3&4&5&6&7&8&9&10&11&12&13&14&15\\\cline{1-17}
$H^k(\widetilde G_{8,3};\mathbb Z)$& $\mathbb Z$ & 0 & 0 & $\mathbb Z_2$ & $\mathbb Z$ & $T_4$ & $T_5$ & $\mathbb Z\oplus T_6$ & $\mathbb Z\oplus T_7$ & $T_6$ & $T_5$ & $\mathbb Z\oplus T_4$ & 0 & $\mathbb Z_2$ & 0 & $\mathbb Z$\\
  \end{tabular}
}
\caption{\label{tab:G83}Integral cohomology groups of $\widetilde G_{8,3}$}
\end{table}

What remains is to determine these four torsion groups.

Recall from Theorem \ref{W2n2,1} the cohomology algebra of $W_{2,1}^8$:
\begin{equation}\label{multiplication_W8}H^*(W^8_{2,1};\mathbb Z)\cong\frac{\mathbb Z[\bar x_2,\bar x_6,\bar x_7]}{\langle \bar x_2^3-2\bar x_{6},\bar x_6^2, \bar x_7^2\rangle}.\end{equation}
 Consider the Leray-Serre spectral sequence for the sphere bundle $S^2\rightarrow W_{2,1}^8\xrightarrow{p} \widetilde G_{8,3}$
where $p$ maps the ordered pair of an oriented plane and a vector to the oriented 3-space they span. Its $E_2$ page is given in Table \ref{tab:G83}.

\begin{center}
\begin{table}[h]
\begin{center}
\scalebox{0.78}{
\renewcommand{\arraystretch}{1.5} %
\begin{tabular}{c|c|c|c|c|c|c|c|c|c|c|c|c|c|c|c|c}
2& $\mathbb Z$ & 0 & 0 & $\mathbb Z_2$ & $\mathbb Z$ & $T_4$ & $T_5$ & $\mathbb Z\oplus T_6$ & $\mathbb Z\oplus T_7$ & $T_6$ & $T_5$ & $\mathbb Z\oplus T_4$ & 0 & $\mathbb Z_2$ & 0 & $\mathbb Z$\\\cline{1-17}
1&&&&&&&&&&&&&&&&\\\cline{1-17}
0& $\mathbb Z$ & 0 & 0 & $\mathbb Z_2$ & $\mathbb Z$ & $T_4$ & $T_5$ & $\mathbb Z\oplus T_6$ & $\mathbb Z\oplus T_7$ & $T_6$ & $T_5$ & $\mathbb Z\oplus T_4$ & 0 & $\mathbb Z_2$ & 0 & $\mathbb Z$\\\cline{1-17}
&0&1&2&3&4&5&6&7&8&9&10&11&12&13&14&15\\
  \end{tabular}
}
\end{center}
\caption{\label{tab:G83}The $E_2$ page for the sphere bundle $S^2\rightarrow W_{2,1}^{8}\rightarrow\widetilde G_{8,3}$}
\end{table}
\end{center}

The only nontrivial differential is $d_3$ so we have $E_2=E_3$ and $E_4=E_\infty$.
Since $H^5(W_{2,1}^8;\mathbb Z)=0$, differential $d_3:E_3^{2,2}\rightarrow E_3^{5,0}$ is onto, so $T_4=0$. Also, for the same reason we have $d_3:E_3^{3,2}\rightarrow E_3^{6,0}$ is one-to-one so $T_5$ contains an element of order two. On the other hand, $H^{12}(W_{2,1}^8;\mathbb Z)=0$ so $d_3:E_3^{10,2}\rightarrow E_3^{13,0}$ is one-to-one and $T_5$ is isomorphic to a subgroup of $\mathbb Z_2$. This gives us $T_5\cong\mathbb Z_2$.  Since $E_3^{5,2}=T_4=0$, differential $d_3:E_3^{5,2}\rightarrow E_3^{8,0}$ is trivial, so $E_\infty^{8,0}\cong\mathbb Z\oplus T_7$, but from (\ref{multiplication_W8}) we see that this group must be free, as $E_\infty^{8,0}$ is a subgroup of the free group $H^8(W_{2,1}^8;\mathbb Z)$, so $T_7=0$.  We have $E_\infty^{9,0}=0$, since it must be free as a subgroup of the free group $H^9(W_{2,1}^8;\mathbb Z)$, so $d_3:E_3^{6,2}\rightarrow E_3^{9,0}$ is onto and since we know $T_5=\mathbb Z_2$, this gives us that $T_6$ is either $0$ or $\mathbb Z_2$.

\begin{table}[h]
\scalebox{0.68}{
\renewcommand{\arraystretch}{1.5} %
\begin{tabular}{c|c|c|c|c|c|c|c|c|c|c|c|c|c|c|c|c}
$k$&0&1&2&3&4&5&6&7&8&9&10&11&12&13&14&15\\\cline{1-17}
$H^k(\widetilde G_{8,3};\mathbb Z)$& $\mathbb Z$ & 0 & 0 & $\mathbb Z_2$ & $\mathbb Z$ & $0$ & $\mathbb Z_2$ & $\mathbb Z\oplus T_6$ & $\mathbb Z$ & $T_6$ & $\mathbb Z_2$ & $\mathbb Z$ & 0 & $\mathbb Z_2$ & 0 & $\mathbb Z$\\\cline{1-17}
$H^k(\widetilde G_{8,3};\mathbb Z_2)$& $\mathbb Z_2$ & 0 & $\mathbb Z_2$ & $\mathbb Z_2$ & $\mathbb Z_2$ & $\mathbb Z_2$ & $\mathbb Z_2\oplus T_6$ & $\mathbb Z_2\oplus T_6$ & $\mathbb Z_2\oplus T_6$ & $\mathbb Z_2\oplus T_6$ & $\mathbb Z_2$ & $\mathbb Z_2$ & $\mathbb Z_2$ & $\mathbb Z_2$ & 0 & $\mathbb Z_2$\\

  \end{tabular}
}
\caption{\label{tab:G83ZZ2} Cohomology groups of $\widetilde G_{8,3}$ with coefficients in $\mathbb Z$ and $\mathbb Z_2$}
\end{table}

Cohomology groups of $\widetilde G_{8,3}$ with coefficients in $\mathbb Z$ and $\mathbb Z_2$ are given in Table \ref{tab:G83ZZ2}. Proposition 3.1 in \cite{korbas} gives us that $H^6(\widetilde G_{8,3};\mathbb Z_2)\cong\mathbb Z_2$, but from Table \ref{tab:G83ZZ2} we see that this group is isomorphic to $\mathbb Z_2\oplus T_6$. This means $T_6=0$.

We have proved so far that the integral cohomology groups of $\widetilde G_{8,3}$ are those given in Table \ref{tab:G83Z}.

\begin{table}[h]
\begin{center}
\scalebox{0.85}{
\renewcommand{\arraystretch}{1.5} %
\begin{tabular}{c|c|c|c|c|c|c|c|c|c|c|c|c|c|c|c|c}
$k$&0&1&2&3&4&5&6&7&8&9&10&11&12&13&14&15\\\cline{1-17}
$H^k(\widetilde G_{8,3};\mathbb Z)$& $\mathbb Z$ & 0 & 0 & $\mathbb Z_2$ & $\mathbb Z$ & $0$ & $\mathbb Z_2$ & $\mathbb Z$ & $\mathbb Z$ & 0 & $\mathbb Z_2$ & $\mathbb Z$ & 0 & $\mathbb Z_2$ & 0 & $\mathbb Z$
  \end{tabular}
}
\caption{\label{tab:G83Z} Integral cohomology groups of $\widetilde G_{8,3}$}
\end{center}
\end{table}

Let us label generators of these groups as $y_3,x_4,y_6,x_7,x_8,y_{10},x_{11},y_{13},x_{15}$ where the index of a generator represents its degree. Some relations are obvious due to dimensions, namely, $y_3^3=0$, $x_4^3=0$, $x_7^2=0$, because these elements are in trivial cohomology groups.
\begin{center}
\begin{table}[h]
\begin{center}
\scalebox{0.9}{
\renewcommand{\arraystretch}{1.5} %
\begin{tabular}{c|c|c|c|c|c|c|c|c|c|c|c|c|c|c|c|c}
2& $1$ &  &  & $y_3$ & $x_4$ &   & $y_6$ & $x_7$ & $x_8$ &   & $y_{10}$ & $x_{11}$ &  & $y_{13}$ &  & $x_{15}$\\\cline{1-17}
1&&&&&&&&&&&&&&&&\\\cline{1-17}
0& $1$ &  &  & $y_3$ & $x_4$ &   & $y_6$ & $x_7$ & $x_8$ &   & $y_{10}$ & $x_{11}$ &  & $y_{13}$ &  & $x_{15}$\\\cline{1-17}
&0&1&2&3&4&5&6&7&8&9&10&11&12&13&14&15\\
  \end{tabular}
}
\caption{\label{tab:G83gen} Generators in the $E_2$ page for the sphere bundle $S^2\rightarrow W_{2,1}^{8}\rightarrow\widetilde G_{8,3}$}
\end{center}
\end{table}
\end{center}

We return now to the spectral sequence for the sphere bundle $S^2\rightarrow W_{2,1}^8\rightarrow \widetilde G_{8,3}$. Since $H^3(W_{2,1}^8;\mathbb Z)=0$, differential $d_3:E_3^{0,2}\rightarrow E_3^{3,0}$ is onto. If we label a generator of $H^2(S^2;\mathbb Z)$ with $s$, we have $d_3(1\otimes s)=y_3\otimes 1$ which then gives us that for any $x\otimes s\in E_3^{k,2}$ 
\[d_3(x\otimes s)=xy_3\otimes 1.\]

 Further, $d_3:E_3^{3,2}\rightarrow E_3^{6,0}$ is an isomorphism and so we get $y_6=y_3^2$. Similarly, $y_{10}=y_3x_7$ and $y_{13}=y_3y_{10}=y_3^2x_7$.

The differential $d_3:E_3^{4,2}\rightarrow E_3^{7,0}$ is multiplication by some integer, but since both groups $H^6(W_{2,1}^8;\mathbb Z)$ and $H^7(W_{2,1}^8;\mathbb Z)$ are isomorphic to $\mathbb Z$, this integer must be zero, i.e., this differential is trivial and so we get 
\begin{equation}\label{eq:y3x4}
y_3x_4=0.
\end{equation}
\indent There is a correspondence between elements of $H^*(\widetilde G_{8,3};\mathbb Z)$ and $H^*(W_{2,1}^8;\mathbb Z)$ which will give us relations between generators of $H^*(\widetilde G_{8,3};\mathbb Z)$  in dimensions $4, 7, 8, 11$ and $15$, where the free part of the cohomology group is nontrivial.
Let us first notice that all differentials $d_3:E_3^{k-3,2}\rightarrow E_3^{k,0}$ are trivial for $k\in\{4,7,8,11,15\}$. We already explained the case $k=7$ when proving the relation (\ref{eq:y3x4}). The same argument works for $k=11$ and the other cases are obvious due to the fact that $E_3^{k-3,2}=0$ for $k\in\{4,8,15\}$. This means that $E_2^{k,0}=E_\infty^{k,0}$, for $k\in\{4,7,8,11,15\}$.\\
\indent Theorem 5.9 in \cite{mccleary} (p. 147) states that the composition
\[H^k(\widetilde G_{8,3};\mathbb Z)=E_2^{k,0}\twoheadrightarrow E_3^{k,0}\twoheadrightarrow \cdots \twoheadrightarrow E_{k+1}^{k,0}\twoheadrightarrow E_\infty^{k,0}\subseteq H^k(W_{2,1}^8;\mathbb Z)\]
is actually the map $p^*:H^k(\widetilde G_{8,3};\mathbb Z)\rightarrow H^k(W_{2,1}^8;\mathbb Z)$. In the case $k\in\{4,7,8,11,15\}$,  $E_2^{k,0}=E_\infty^{k,0}$, so we have that in these dimensions $p^*$ is an injection.\\
\indent The Leray-Serre spectral sequence comes with the filtration of the cohomology groups of the total space
\[0\subseteq F^pH^p\subseteq F^{p-1}H^p\subseteq\cdots\subseteq F^1H^p\subseteq F^0H^p=H^p(W_{2,1}^8;\mathbb Z),\]
and $E_\infty^{r,p-r}\cong F^rH^p/F^{r+1}H^p$.\\
\indent Let $k\in \{4, 7, 11\}$. We have $E_\infty^{k,0}=F^kH^k=F^0H^k=H^k(W_{2,1}^8;\mathbb Z)$, which implies that $p^*:H^k(\widetilde G_{8,3};\mathbb Z)\rightarrow H^k(W_{2,1}^8;\mathbb Z)$ is an isomorphism, and so the generators $x_4, x_7$ and $x_{11}$ can be chosen in such a way that
\begin{equation}\label{x4711}
p^*(x_4)=\bar x_2^2,~p^*(x_7)=\bar x_7,~p^*(x_{11})=\bar x_2^2\bar x_7.
\end{equation}

On the other hand, if $k\in\{8,15\}$ we have 
\[E_\infty^{k,0}=F^kH^k(W_{2,1}^8;\mathbb Z)=F^{k-1}H^k=2F^{k-2}H^k=2F^0H^k=2H^k(W_{2,1}^8;\mathbb Z),\]
since $E_\infty^{k-2,2}\cong F^{k-2}H^k/F^{k-1}H^k\cong\mathbb Z_2$. The generators $x_8$ and $x_{15}$ can now be chosen with properties
\begin{equation}\label{x815}
p^*(x_8)=2\bar x_2\bar x_6,~p^*(x_{15})=2\bar x_2\bar x_6\bar x_7.
\end{equation}
\indent From relations in the algebra (\ref{multiplication_W8}) and equations (\ref{x4711}) and (\ref{x815}), we get
\[\begin{split}p^*(x_8)&=2\bar x_2\bar x_6=\bar x_2^4=p^*(x_4^2),\\
p^*(x_{11})&=\bar x_2^2\bar x_7=p^*(x_4x_7),\\
p^*(x_{15})&=2\bar x_2\bar x_6\bar x_7=\bar x_2^4\bar x_7=p^*(x_4^2x_7).
\end{split}\]
\indent Since $p^*$ is an injection in dimensions 8, 11 and 15, we have the following relations in $H^*(\widetilde G_{8,3};\mathbb Z)$:
\[x_8=x_4^2,~x_{11}=x_4x_7,~x_{15}=x_4^2x_7,\]
which, in addition to previous relations, give the specified structure of the algebra $H^*(\widetilde G_{8,3};\mathbb Z)$.
\end{proof}

\begin{remark}
If $\rho:\mathbb Z\rightarrow\mathbb Z_2$ is the modulo 2 map, then
\[\rho_*(y_3)=\widetilde w_3,~\rho_*(x_4)=\widetilde w_2^2,~\rho_*(x_7)=a_7,\]
where $\widetilde w_2, \widetilde w_3$ and $a_7$ are generators of $H^*(\widetilde G_{8,3};\mathbb Z_2)$ from Proposition 3.1 in \cite{korbas}. This is proven by using the universal coefficient theorem in the same way as in the proof of Theorem \ref{G2n+1,2}.
\end{remark}

\section{Cohomology algebra $H^*(\widetilde G_{10,3};\mathbb Z)$}
\label{G103}

In the following theorem we will prove that the integral cohomology algebra of $\widetilde G_{10,3}$ is isomorphic to the quotient of the polynomial algebra in five variables with coefficients in $\mathbb Z$ by the sum of ideals $\mathcal I$ and $\mathcal J$. Since the dimension of $\widetilde G_{10,3}$ is 21, we know that $H^k(\widetilde G_{10,3};\mathbb Z)=0$ for $k>21$, so the ideal $\mathcal J$ will contain all monomials of (cohomological) degree larger than 21. The ideal $\mathcal I$ will be generated by all homogeneous polynomials corresponding to the relations in cohomology groups up to dimension 21.

\begin{theorem}\label{ringG103}
The integral cohomology algebra of $\widetilde G_{10,3}$  is given by
\[H^*(\widetilde G_{10,3};\mathbb Z)\cong \frac{\mathbb Z[y_3,x_4,x_9,x_{12},x_{13}]}{\mathcal I+\mathcal J},\]
where $y_3\in H^3(\widetilde G_{10,3};\mathbb Z)$, $x_i\in H^i(\widetilde G_{10,3};\mathbb Z)$, for $i\in\{4,9,12,13\}$, the ideal $\mathcal I$  is generated by polynomials:
\[\begin{split}
&2y_3,~ y_3^3,~  y_3^2x_4,~  y_3x_4^2,~  x_4^3-2x_{12},~ y_3x_9,~  x_4x_9-2x_{13},\\  
&y_3x_{13}-x_4x_{12},~ x_9^2,~ y_3^2x_{12},~  x_4^2x_{12},~  x_9x_{12}-x_4^2x_{13};
\end{split}\]
and the ideal $\mathcal J$ is generated by monomials:
\[y_3^ax_4^bx_9^cx_{12}^dx_{13}^e~ \text{such that}~ 3a+4b+9c+12d+13e>21.\]
\end{theorem}
\begin{proof}  From \cite{greub} (p. 496), the Poincar\' e polynomial for $\widetilde G_{10,3}$ is
\[p(\widetilde G_{10,3})=1+x^4+x^8+x^9+x^{12}+x^{13}+x^{17}+x^{21},\]
so we know the free part of the cohomology groups of $\widetilde G_{10,3}$.

When we apply results from Section \ref{propertiesGn3}, we see that integral cohomology groups of $\widetilde G_{10,3}$ are:
\[\begin{split}
H^*(\widetilde G_{10,3};\mathbb Z)=(&\mathbb Z,0,0,\mathbb Z_2,\mathbb Z,T_4,T_5,T_6, \mathbb Z\oplus T_7,\mathbb Z\oplus T_8,T_9, T_{10}, \mathbb Z\oplus T_9, \\
&\mathbb Z\oplus T_8,T_7,T_6,T_5,\mathbb Z \oplus T_4,0,\mathbb Z_2,0,\mathbb Z).
\end{split}\]

What remains is to determine these seven torsion groups.

Recall from Theorem \ref{W2n2,1} the cohomology algebra of $W_{2,1}^{10}$
\begin{equation}
\label{multiplication_W10}
H^*(W^{10}_{2,1};\mathbb Z)\cong\frac{\mathbb Z[\bar x_2,\bar x_8,\bar x_9]}{\langle \bar x_2^4-2\bar x_8,\bar x_8^2, \bar x_9^2\rangle}.
\end{equation}
 Consider the Leray-Serre spectral sequences for the sphere bundle $S^2\rightarrow W_{2,1}^{10}\xrightarrow{p} \widetilde G_{10,3}$ and for the fiber bundle $SO(3)\rightarrow V_{10,3}\rightarrow \widetilde G_{10,3}$. The cohomology ring $H^*(V_{10,3};\mathbb Z)$ is known from Theorem 2.3 in \cite{mimura}.

In order to differentiate these spectral sequences, we will label the first one as $E$ and the second one as $\bar E$. Arguments similar as in the proof of Theorem \ref{ringG83} give us $T_4=0$, $T_5=\mathbb Z_2$ and $T_7=0$. Since $H^9(W_{2,1}^{10};\mathbb Z)\cong \mathbb Z$, $E_\infty^{9,0}$ must be free, so $d_3:E_3^{6,2}\rightarrow E_3^{9,0}$ from Table \ref{tab:G103E1} maps $\mathbb Z_2$ onto $T_8$, i.e., $T_8$ is either $\mathbb Z_2$ or trivial. 
From Table \ref{tab:G103E1} we also see that $H^8(\widetilde G_{10,3};\mathbb Z_2)\cong \mathbb Z_2\oplus T_8$, but from Proposition 3.1 in \cite{korbas} this group is isomorphic to $\mathbb Z_2$, so we get $T_8=0$.

\begin{table}[h]
\begin{center}
\scalebox{0.65}{
\renewcommand{\arraystretch}{1.5} %
\begin{tabular}{c|c|c|c|c|c|c|c|c|c|c|c|c|c|c|c|c|c|c|c|c|c|c}
2& $\mathbb Z$ & 0 & 0 & $\mathbb Z_2$ & $\mathbb Z$ & $0$ & $\mathbb Z_2$ & $T_6$ & $\mathbb Z$ & $\mathbb Z\oplus T_8$ & $T_9$ & $T_{10}$ & $\mathbb Z \oplus T_9$ & $\mathbb Z\oplus T_8$ & $0$ & $T_6$ & $\mathbb Z_2$ & $\mathbb Z$ & $0$ & $\mathbb Z_2$ & $0$ & $\mathbb Z$\\\cline{1-23}
1&&&&&&&&&&&&&&&&&&&&&&\\\cline{1-23}
0& $\mathbb Z$ & 0 & 0 & $\mathbb Z_2$ & $\mathbb Z$ & $0$ & $\mathbb Z_2$ & $T_6$ & $\mathbb Z$ & $\mathbb Z\oplus T_8$ & $T_9$ & $T_{10}$ & $\mathbb Z \oplus T_9$ & $\mathbb Z\oplus T_8$ & $0$ & $T_6$ & $\mathbb Z_2$ & $\mathbb Z$ & $0$ & $\mathbb Z_2$ & $0$ & $\mathbb Z$\\\cline{1-23}
&0&1&2&3&4&5&6&7&8&9&10&11&12&13&14&15&16&17&18&19&20&21\\
  \end{tabular}
}
\end{center}
\caption{\label{tab:G103E1}The $E_2$ page for the sphere bundle $S^2\rightarrow W_{2,1}^{10}\rightarrow\widetilde G_{10,3}$}
\end{table}

Since $H^{18}(V_{10,3};\mathbb Z)=0$ \cite{mimura}, Table \ref{tab:G103E2} gives us that differential $d_3:\bar E_3^{16,2}\rightarrow \bar E_3^{19,0}$ is an isomorphism and $\bar E_4^{19,0}=0$. This means that $d_2:\bar E_2^{15,3}\rightarrow \bar E_2^{17,2}$ is a monomorphism, as this is the only nontrivial differential with domain $E_k^{15,3}$. On the other hand, since $H^{19}(V_{10,3};\mathbb Z)=0$  \cite{mimura}, this differential must also be an epimorphism. So we get $T_6\cong \mathbb Z_2$.
\begin{table}[h]
\begin{center}
\scalebox{0.8}{
\renewcommand{\arraystretch}{1.5} %
\begin{tabular}{c|c|c|c|c|c|c|c}
3  &  $0$ & $T_6$ & $\mathbb Z_2$ & $\mathbb Z$ & $0$ & $\mathbb Z_2$ & 0\\\cline{1-8}
2& $\mathrm{Tor}(T_6,\mathbb Z_2)$ & $T_6\otimes \mathbb Z_2\oplus \mathbb Z_2$ & $\mathbb Z_2$ & $\mathbb Z_2$ & $\mathbb Z_2$ & $\mathbb Z_2$ &  0\\\cline{1-8}
1&&&&&&&\\\cline{1-8}
0 &   $0$ & $T_6$ & $\mathbb Z_2$ & $\mathbb Z$ & $0$ & $\mathbb Z_2$ & 0\\\cline{1-8}
&14&15&16&17&18&19&20\\
  \end{tabular}
}
\caption{\label{tab:G103E2}The part of the $\bar E_2$ page for the fiber bundle $SO(3)\rightarrow V_{10,3}\rightarrow\widetilde G_{10,3}$}
\end{center}
\end{table}

If we return to Table \ref{tab:G103E1}, differential $d_3:E_3^{7,2}\rightarrow E_3^{10,0}$ must be onto since $H^{10}(W_{2,1}^{10};\mathbb Z)$ is free. This means that $T_9$ is either $\mathbb Z_2$ or trivial, as now we know that $E_3^{7,2}\cong T_6\cong\mathbb Z_2$. From Table \ref{tab:G103E1} we get that $H^9(\widetilde G_{10,3};\mathbb Z_2)\cong \mathbb Z_2\oplus T_9$, but from Proposition 3.1 in \cite{korbas} this group is isomorphic to $\mathbb Z_2$, so it must be $T_9=0$.

What remains is to determine the torsion group $T_{10}$. In Table \ref{tab:G103E3} the other part of the second spectral sequence is given. Since $H^{14}(V_{10,3};\mathbb Z)=0$  \cite{mimura}, differential $d_3:\bar E_3^{12,2}\rightarrow\bar E_3^{15,0}$ is an isomorphism, so the only possible nontrivial differential with domain $\bar E_k^{11,3}$ is $d_2:\bar E_2^{11,3}\rightarrow \bar E_2^{13,2}$ and it must be a monomorphism because of the same argument $H^{14}(V_{10,3};\mathbb Z)=0$. This gives us that $T_{10}$ is isomorphic to a subgroup of $\mathbb Z_2$. When we combine this result with the fact that $H^{10}(\widetilde G_{10,3};\mathbb Z_2)=0$ from Proposition 3.1 in \cite{korbas}, we get $T_{10}=0$.

\begin{table}[h]
\begin{center}
\scalebox{0.8}{
\renewcommand{\arraystretch}{1.5} %
\begin{tabular}{c|c|c|c|c|c|c}
3  &  $0$ & $T_{10}$ & $\mathbb Z$ & $\mathbb Z$ & $0$ & $\mathbb Z_2$\\\cline{1-7}
2  &  $\mathrm{Tor}(T_{10},\mathbb Z_2)$ & $T_{10}\otimes \mathbb Z_2$ & $\mathbb Z_2$ & $\mathbb Z_2$ & $\mathbb Z_2$ & $\mathbb Z_2\oplus\mathbb Z_2$\\\cline{1-7}
1&&&&&&\\\cline{1-7}
0   &  $0$ & $T_{10}$ & $\mathbb Z$ & $\mathbb Z$ & $0$ & $\mathbb Z_2$\\\cline{1-7}
&10&11&12&13&14&15\\
  \end{tabular}
}
\end{center}
\caption{\label{tab:G103E3}The part of the $\bar E_2$ page for the fiber bundle $SO(3)\rightarrow V_{10,3}\rightarrow\widetilde G_{10,3}$}
\end{table}

Finally, we have integral cohomology groups of $\widetilde G_{10,3}$:

\begin{table}[h]
\scalebox{0.675}{
\renewcommand{\arraystretch}{1.5} %
\begin{tabular}{c|c|c|c|c|c|c|c|c|c|c|c|c|c|c|c|c|c|c|c|c|c|c}
$k$&0&1&2&3&4&5&6&7&8&9&10&11&12&13&14&15&16&17&18&19&20&21\\\cline{1-23}
$H^k(\widetilde G_{10,3};\mathbb Z)$& $\mathbb Z$ & 0 & 0 & $\mathbb Z_2$ & $\mathbb Z$ & $0$ & $\mathbb Z_2$ & $\mathbb Z_2$ & $\mathbb Z$ & $\mathbb Z$ & 0 & 0 & $\mathbb Z$ & $\mathbb Z$ & $0$ & $\mathbb Z_2$ & $\mathbb Z_2$ & $\mathbb Z$ & $0$ & $\mathbb Z_2$ & $0$ & $\mathbb Z$\\
  \end{tabular}
}
\caption{\label{tab:G103groups}Integral cohomology groups of $\widetilde G_{10,3}$}
\end{table}

Let us label generators of these groups as $y_3, x_4, y_6, y_7, x_8, x_9, x_{12},x_{13},$ $y_{15}, y_{16}, x_{17}, y_{19}, x_{21}$ where the index of a generator represents its degree.  In the same way as in the proof of Theorem \ref{ringG83}, we conclude that the only nontrivial differential $d_3:E_3^{k,2}\rightarrow E_3^{k+3,0}$ is given by
\[d_3(x\otimes s)=xy_3\otimes 1,\]
for $x\in H^k(\widetilde G_{10,3};\mathbb Z)$ and that the following relations hold:
\[y_6=y_3^2, y_7=y_3x_4, y_{15}=y_3x_{12}, y_{16}=y_3x_{13}, y_{19}=y_3^2x_{13}.\]
\indent We will now use Theorem 5.9 in \cite{mccleary} (p. 147) in the same way as in the proof of Theorem \ref{ringG83} in order to determine relations between generators.
\begin{center}
\begin{table}[h]
\begin{center}
\scalebox{0.7}{
\begingroup
\renewcommand{\arraystretch}{1.5} %
\begin{tabular}{c|c|c|c|c|c|c|c|c|c|c|c|c|c|c|c|c|c|c|c|c|c|c}
2& $\mathbb Z$\tikzmark{g} & 0 & 0 & $\mathbb Z_2$\tikzmark{i} & $\mathbb Z$\tikzmark{k} & $0$ & $\mathbb Z_2$ & $\mathbb Z_2$ & $\mathbb Z$ & $\mathbb Z$ & $0$ & $0$ & $\mathbb Z$\tikzmark{m} & $\mathbb Z$\tikzmark{o} & $0$ & $\mathbb Z_2$ & $\mathbb Z_2$\tikzmark{q} & $\mathbb Z$ & $0$ & $\mathbb Z_2$ & $0$ & $\mathbb Z$\\\cline{1-23}
1&&&&&&&&&&&&&&&&&&&&&&\\\cline{1-23}
0& $\mathbb Z$ & 0 & 0 & \tikzmark{h}$\mathbb Z_2$ & $\mathbb Z$ & $0$ & \tikzmark{j}$\mathbb Z_2$ & \tikzmark{l}$\mathbb Z_2$ & $\mathbb Z$ & $\mathbb Z$ & $0$ & $0$ & $\mathbb Z$ & $\mathbb Z$ & $0$ & \tikzmark{n}$\mathbb Z_2$ & \tikzmark{p}$\mathbb Z_2$ & $\mathbb Z$ & $0$ & \tikzmark{r}$\mathbb Z_2$ & $0$ & $\mathbb Z$\\\cline{1-23}
&\bf{0}&1&\bf{2}&3&\bf{4}&5&\bf{6}&7&\bf{8}&\bf{9}&\bf{10}&\bf{11}&\bf{12}&\bf{13}&\bf{14}&\bf{15}&16&\bf{17}&18&\bf{19}&20&\bf{21}\\
  \end{tabular}

  \begin{tikzpicture}[overlay, remember picture, shorten >=.5pt, shorten <=.5pt, transform canvas={yshift=.25\baselineskip}]
    \draw [->>] ([yshift=.75pt]{pic cs:g}) -- ({pic cs:h});
    \draw [>->>] ([yshift=.75pt]{pic cs:i}) -- ({pic cs:j});
    \draw [->>] ([yshift=.75pt]{pic cs:k}) -- ({pic cs:l});
    \draw [->>] ([yshift=.75pt]{pic cs:m}) -- ({pic cs:n});
    \draw [->>] ([yshift=.75pt]{pic cs:o}) -- ({pic cs:p});
    \draw [>->>] ([yshift=.75pt]{pic cs:q}) -- ({pic cs:r});
  \end{tikzpicture}
\endgroup
}
\caption{\label{tab:G103svegr}The $E_3$ page for the sphere bundle $S^2\rightarrow W_{2,1}^{10}\rightarrow \widetilde G_{10,3}$}
\end{center}
\end{table}
\end{center}

 In the case $k\in\{4,8,9,12,13,17,21\}$ we see from Table \ref{tab:G103svegr} (where all but trivial differentials are represented by arrows and dimensions of nontrivial cohomology groups of $W_{2,1}^{10}$ are in bold) that $E_2^{k,0}=E_\infty^{k,0}$, so we have that in these dimensions $p^*$ is an injection.
\\
\indent Let $k\in \{4, 12, 13\}$. We have $E_\infty^{k,0}=F^kH^k=F^0H^k=H^k(W_{2,1}^{10};\mathbb Z)$, so $p^*$ maps a generator of $H^k(\widetilde G_{10,3};\mathbb Z)$ to the specified (in (\ref{multiplication_W10})) generator of $H^k(W_{2,1}^{10};\mathbb Z)$, i.e., we can choose generators $x_4, x_{12}$ and $x_{13}$ to have
\begin{equation}\label{x41213}
p^*(x_4)=\bar x_2^2,~p^*(x_{12})=\bar x_2^2\bar x_8,~p^*(x_{13})=\bar x_2^2\bar x_9.
\end{equation}
\indent On the other hand, if $k\in\{8,9,17,21\}$ we have 
\[E_\infty^{k,0}=F^kH^k(W_{2,1}^{10};\mathbb Z)=F^{k-1}H^k=2F^{k-2}H^k=2F^0H^k=2H^k(W_{2,1}^{10};\mathbb Z),\]
since $E_\infty^{k-2,2}\cong F^{k-2}H^k/F^{k-1}H^k\cong\mathbb Z_2$. So we choose $x_8, x_9, x_{17}$ and $x_{21}$ with properties
\begin{equation}\label{x891721}
p^*(x_8)=2\bar x_8,~p^*(x_9)=2\bar x_9,~p^*(x_{17})=2\bar x_8\bar x_9,~p^*(x_{21})=2\bar x_2^2\bar x_8\bar x_9.
\end{equation}
\indent From relations in the algebra (\ref{multiplication_W10}) and equations (\ref{x41213}) and (\ref{x891721}), we get
\[\begin{split}p^*(x_8)&=2\bar x_8=\bar x_2^4=p^*(x_4^2),\\
p^*(2x_{12})&=2\bar x_2^2\bar x_8=\bar x_2^6=p^*(x_4^3),\\
p^*(2x_{13})&=2\bar x_2^2\bar x_9=p^*(x_4x_9),\\
p^*(x_{17})&=2\bar x_8\bar x_9=\bar x_2^2\bar x_2^2\bar x_9=p^*(x_4x_{13}),\\
p^*(x_{21})&=2\bar x_2^2\bar x_8\bar x_9=\bar x_2^4\bar x_2^2\bar x_9=p^*(x_4^2x_{13})=p^*(x_9x_{12}).
\end{split}\]
\indent Since $p^*$ is an injection in these dimensions, we have the following relations in $H^*(\widetilde G_{10,3};\mathbb Z)$:
\[x_8=x_4^2,~2x_{12}=x_4^3,~2x_{13}=x_4x_9,~x_{17}=x_4x_{13},~x_{21}=x_4^2x_{13}=x_9x_{12}.\]
\indent This means that cohomology algebra $H^*(\widetilde G_{10,3};\mathbb Z)$ is generated by five generators $y_3, x_4, x_9, x_{12}$ and $x_{13}$.
\\
\indent Some relations are obvious due to dimension. We have $y_3^2x_4=0$, $y_3x_4^2=0$, $x_9^2=0$, $y_3^2x_{12}=0$, $x_4^2x_{12}=y_3x_4x_{13}=0$ since these elements are in trivial cohomology groups. Other relations are obvious since cup product of $y_3$ and any other element must be trivial in torsion free groups, so $y_3^3=0$ and $y_3x_9=0$. Also, $y_3^ax_4^bx_9^cx_{12}^dx_{13}^e=0$ whenever $3a+4b+9c+12d+13e>21$ since the dimension of $\widetilde G_{10,3}$ is 21.
\\
\indent What remains is to prove the relation $y_3x_{13}=x_4x_{12}$. We will use naturality of the universal coefficient theorem and Proposition 3.1 in \cite{korbas}. This proposition gives us
\[H^4(\widetilde G_{10,3};\mathbb Z_2)=\mathbb Z_2\langle \widetilde w_2^2\rangle, H^{12}(\widetilde G_{10,3};\mathbb Z_2)=\mathbb Z_2\langle  a_{12}\rangle, H^{16}(\widetilde G_{10,3};\mathbb Z_2)=\mathbb Z_2\langle \widetilde w_2^2a_{12}\rangle.\]
\indent We have the following commutative diagram:
\[
\begin{tikzcd}[column sep = tiny, font = \small]
0
\arrow{r}
&
\mathrm{Ext}(H_{k-1}(\widetilde G_{10,3}),\mathbb Z)
\arrow{d}{\rho_*}
\arrow{r}
&
H^k(\widetilde G_{10,3};\mathbb Z)
\arrow{d}{\rho_*}
\arrow{r}
&
\mathrm{Hom}(H_k(\widetilde G_{10,3}),\mathbb Z)
\arrow{d}{\rho_*}
\arrow{r}
&
0
\\
0
\arrow{r}
&
\mathrm{Ext}(H_{k-1}(\widetilde G_{10,3}),\mathbb Z_2)
\arrow{r}
&
H^k(\widetilde G_{10,3};\mathbb Z_2)
\arrow{r}
&
\mathrm{Hom}(H_k(\widetilde G_{10,3}),\mathbb Z_2)
\arrow{r}
&
0
\end{tikzcd}
\]

For $k=4$ both $\mathrm{Ext}$-groups are zero and $\rho_*:\mathrm{Hom}(H_4(\widetilde G_{10,3}),\mathbb Z)\rightarrow\mathrm{Hom}(H_4(\widetilde G_{10,3}),\mathbb Z_2)$ is an epimorphism, so $\rho_*:H^4(\widetilde G_{10,3};\mathbb Z)\rightarrow H^4(\widetilde G_{10,3};\mathbb Z_2)$ is an epimorphism, i.e., $\rho_*(x_4)=\widetilde w_2^2$. The situation is the same for $k=12$, so $\rho_*(x_{12})=a_{12}$. Now we have $\rho_*(x_4x_{12})=\widetilde w_2^2a_{12}\ne 0$, so $x_4x_{12}\ne 0$, i.e., $x_4x_{12}=y_3x_{13}$.
\end{proof}

\begin{remark}
    Since $\mathbb Z[y_3,x_4,x_9,x_{12},x_{13}]$ is a Noetherian ring, it is possible to find a finite generating set for the ideal $\mathcal I + \mathcal J$ from Theorem \ref{ringG103}. One can do this by starting with the generators of $\mathcal I$, then going through all elements $y_3^ax_4^bx_9^cx_{12}^dx_{13}^e$ with $3a+4b+9c+12d+13e = k$ for $k=22,23,24,\ldots$ to see which elements are already in $\mathcal I$, and then adding those that are not. It suffices to do this checking up to dimension $k=34$ since all the monomials of degree greater than $34$ are multiples of those in degrees $22\leqslant k \leqslant 34$ (this is because the maximal degree of generators $y_3, x_4, x_9, x_{12}$ and $x_{13}$ is $13$). After going through this procedure, we obtain that there are only a few elements that need to be added to the generating set of $\mathcal I$ in order to get the whole ideal $\mathcal I + \mathcal J$, namely $x_9x_{13}, x_{12}^2, x_{12}x_{13}$ and $x_{13}^2$. Finally, we can write the integral cohomology algebra of $\widetilde G_{10,3}$ as
    \[H^*(\widetilde G_{10,3};\mathbb Z) \cong \frac{\mathbb Z[y_3,x_4,x_9,x_{12},x_{13}]}{\mathcal K},\]
    where $\mathcal K$ is generated by polynomials:
    \[\begin{split}
&2y_3,~ y_3^3,~  y_3^2x_4,~  y_3x_4^2,~  x_4^3-2x_{12},~ y_3x_9,~  x_4x_9-2x_{13},\\  
&y_3x_{13}-x_4x_{12},~ x_9^2,~ y_3^2x_{12},~  x_4^2x_{12},~  x_9x_{12}-x_4^2x_{13},\\
&x_9x_{13},~ x_{12}^2,~ x_{12}x_{13},~ x_{13}^2.
\end{split}\]
\end{remark}

\begin{remark}
If $\rho:\mathbb Z\rightarrow\mathbb Z_2$ is the modulo 2 map, then
\[\rho_*(y_3)=\widetilde w_3,~\rho_*(x_4)=\widetilde w_2^2,~\rho_*(x_9)=\widetilde w_2^3\widetilde w_3,~\rho_*(x_{12})=a_{12},~\rho_*(x_{13})=a_{13},\]
where $\widetilde w_2, \widetilde w_3, a_{12}$ and $a_{13}$ are generators of $H^*(\widetilde G_{10,3};\mathbb Z_2)$ from Proposition 3.1 in \cite{korbas}. Two of these relations we have proved in the proof of Theorem \ref{ringG103} and the other three are proven in the same way by using the universal coefficient theorem and Poincar\' e duality.
\end{remark}

\section{Acknowledgements}

This research was supported by the Science Fund of the Republic of Serbia,  Grant No. 7749891, Graphical Languages
- GWORDS.






\end{document}